\documentclass[12pt]{amsart}
\usepackage{amsmath}
\usepackage{amsthm}
\usepackage{amsfonts}
\usepackage{amssymb}
\usepackage[T1]{fontenc}
\usepackage{hyperref}
\usepackage[capitalise]{cleveref}
\usepackage{esint}
\usepackage{mathrsfs, mathtools, bm}
\usepackage{xcolor}
\usepackage[margin=2.5cm]{geometry}
\usepackage{cancel}
\usepackage{subfigure}
\usepackage{subcaption}
\usepackage{physics} 
\usepackage{comment}
\usepackage[dvipsnames]{xcolor}
\usepackage[normalem]{ulem}
\usepackage[isbn=false,url=false,eprint=false,style=alphabetic,doi=false]{biblatex}

\newtheorem{theorem}{Theorem}[section]
\newtheorem{proposition}[theorem]{Proposition}
\newtheorem{lemma}[theorem]{Lemma}

\newtheorem{defn}[theorem]{Definition}
\newtheorem{remark}[theorem]{Remark}

\newcommand{\R}{\mathbb{R}}

\newcommand{\nocontentsline}[3]{}
\newcommand{\tocless}[2]{\bgroup\let\addcontentsline=\nocontentsline#1{#2}\egroup}

\title{Compact Mean Curvature Flows from Non-Compact Initial Data}
\author{Theodora Bourni}
\address{Department of Mathematics, University of Tennessee Knoxville,
Knoxville TN, 37996-1320}
\email{tbourni@utk.edu}
\author{Nathan Burns}
\address{Department of Mathematics, University of Tennessee Knoxville,
Knoxville TN, 37996-1320}
\email{nburns6@vols.utk.edu}
\author{Marcus Flook}
\address{Mathematical Sciences Institute, Australian National University, Canberra ACT 2601, Australia}
\email{marcus.flook@anu.edu.au}
\date{}
\begin{document}

\begin{abstract}
    We construct compact, embedded solutions to the mean curvature flow in $\R^{n+1}$ that arise from various non-compact initial data and exist for a finite time.
\end{abstract}

\maketitle

\tableofcontents

\section{Introduction}

Let $F_{0}:M^{n} \to \R^{n+1}$ be an immersion of an $n$-dimensional manifold without boundary into $\R^{n+1}$, and write $\mathcal{M}_{0} \doteqdot F_0(M^{n})$. We are interested in solutions to the mean curvature flow whose initial data is given by $F_{0}$; that is, maps $F:M^{n} \times [0,T) \to \R^{n+1}$ satisfying 
\begin{equation}\label{eq:mcf}
    \begin{cases}
        \displaystyle\frac{\partial}{\partial t}F(x,t) = \vec{H}(x,t) & (x,t) \in M^{n} \times [0,T) \\
        F(x,0) = F_{0}(x) & x \in M^{n},
    \end{cases}
\end{equation}
where $\vec{H}$ denotes the mean curvature vector of the evolving immersed manifold. Such a solution establishes a one parameter family of hypersurfaces $\{\mathcal{M}_{t}\}_{t \in [0,T)}$, where $\mathcal{M}_{t} \doteqdot F(M^{n},t)$, and, by a slight abuse of terminology, we also refer to this family as a solution to mean curvature flow.
The existence theory for \eqref{eq:mcf} in the case that $\mathcal{M}_{0}$ is compact is well understood. After choosing an appropriate gauge, it can be shown that \eqref{eq:mcf} is equivalent to a quasilinear parabolic equation (see \cite[Section 6.4]{AndrewsChowGuentherLangford2020}). The general theory of such equations then yields the existence of a unique solution at least for a short time \cite{Lieberman1996}.

When the initial immersed hypersurface $\mathcal{M}_{0}$ is non-compact, the situation is considerably more subtle. In \cite{EckerHuisken1991}, Ecker and Huisken proved that if $\mathcal{M}_{0}\doteqdot F_{0}(M^{n})$ is given as a locally Lipschitz entire graph over a hyperplane, then \eqref{eq:mcf} admits a smooth solution for all time with each $\mathcal{M}_{t}\doteqdot F(M^{n},t)$ remaining an entire graph. Moreover, in \cite{EckerHuisken1989}, they showed that if the initial hypersurface is Lipschitz with bounded gradient, then the corresponding solution to \eqref{eq:mcf} is unique. This picture was further developed by Daskolopoulos and Saez in \cite{DaskalopoulosSaez2023}. They proved uniqueness for entire graphical solutions with locally Lipschitz initial data in the case $n = 1$ (curve shortening flow) and, in higher dimensions, when the initial hypersurface is rotationally symmetric. They also established uniqueness for entire graphical solutions with a uniform lower bound on the second fundamental form. In particular, this class includes convex entire graphs, and it was subsequently used in \cite{BourniLangfordLynch2023} to prove uniqueness for general convex, locally uniformly convex solutions to mean curvature flow. 

Outside of these relatively rigid regimes, uniqueness, in general, fails and indeed there are examples of distinct solutions to \eqref{eq:mcf} with the same initial hypersurface, see, for example, \cite{AngenentIlmanenChopp1995}.

In this paper we study the possibility of constructing compact mean curvature flows with non-compact initial data. More precisely, let
 $\mathcal{M}_{0} \subset \R^{n+1}$ be a non-compact hypersurface. We seek a family of immersions $F:M^{n} \times (0,T) \to \R^{n+1}$ of a \emph{compact} manifold $M^{n}$ into $\R^{n+1}$ such that
\[
\frac{\partial}{\partial t}F(x,t) = \vec{H}(x,t) \,,\,\,\forall  (x,t) \in M^{n} \times (0,T)
\]
and, as $t\to 0$,  $\mathcal{M}_{t}$ converges to $\mathcal{M}_{0}$, in the following sense:
\begin{equation*}\label{C}
\tag{$\star$}
\begin{split}
\forall \varepsilon>0\,,\,\, \exists t_{\varepsilon} >0, \text{ such that }
\mathcal{M}_{t} \subseteq &B(\mathcal{M}_{0},\varepsilon)\,,\,\, \forall  t \in (0,t_{\varepsilon}),\\
\text{
where }
    B(\mathcal{M}_{0},\varepsilon) \doteqdot \{p \in \R^{n+1} &: \text{dist}(p,\mathcal{M}_{0}) <\varepsilon\}.
    \end{split}
\end{equation*}

\begin{remark}
    Although \eqref{C} and local uniform convergence are not equivalent, it will follow directly from the proofs of Theorems \ref{thm:main1} and \ref{thm:main2} that the constructed mean curvature flows will satisfy $\mathcal{M}_{t} \to \mathcal{M}_{0}$ locally uniformly in $\R^{n+1}$ as $t \to 0$.
\end{remark}

A solution of this form was first described by Ilmanen in \cite[Example 7.3]{Ilmanen1992} in the case $n = 1$, where it was used to demonstrate the non-uniqueness of the level set flow for non-compact sets. More detailed constructions in the case $n=1$ were later given by Bourni and Reiris in \cite{BourniReiris2024}. They proved the existence of a solution in the above sense under the assumptions that $\mathcal{M}_{0}$ is embedded, encloses a region of finite area $A_{0}$, and, sufficiently far out, can be represented as the union of graphs of two monotone functions that converge to zero at infinity, one from above and the other from below. These solutions converge to a round point at time $
T=\frac{A_{0}}{2\pi}$.
A different construction was also given by Peachey \cite{Peachey2022}, who constructed a smooth solution
$
\gamma:S^1\times(0,\infty)\to\mathbb{R}^{2}
$
emerging from a cusp asymptotic to $\pm(1+cx)^{-2}$. Since these examples have finite total curvature, the work of Chou and Zhu \cite{ChouZhu1998} implies the existence of a non-compact solution in the sense of \eqref{eq:mcf} that exists for a finite time
$
T=\frac{A_{0}}{\pi}$.

In this paper, we extend the construction done in \cite{BourniReiris2024} to higher dimensions for a class of rotationally symmetric initial conditions.
Roughly speaking, these initial conditions are obtained by taking a planar curve in the $x-y$ plane that encloses a finite area and rotating it about different axes. We refer to the resulting initial configurations as the \emph{raindrop}, \emph{UFO}, and \emph{stretched torus}. Their precise definitions are given in Definitions \ref{def:raindropIC}, \ref{def:UFOIC} and  \ref{def:stretchedtorusIC}, respectively. 
For the case $n=2$, these configurations are illustrated in Figures \ref{fig:foobar} and \ref{fig:storusR3}, respectively.

Our main results are the following.

\begin{theorem}\label{thm:main1}
    Let $\mathcal{M}_{0}$ be either a \emph{raindrop} or a \emph{UFO} hypersurface in $\R^{n+1}$.
    Then there exists a smooth mean curvature flow
    \[
        F:S^{n} \times (0,T)\to\R^{n+1},
    \]
    such that $\mathcal{M}_t$ converges to $\mathcal{M}_0$ as $t\to0$ in the sense of \eqref{C} and $\mathcal{M}_{t}$ converges to a point on the $x$-axis as $t \to T$. Additionally, $\tfrac{h^2}{2n}\leq T\leq\tfrac{A_0}{2\pi}$, where $h$ is the radius of a sphere enclosed by the initial condition and $A_0$ is the area enclosed by the profile curve of the initial condition (see also Definitions \eqref{def:raindropIC} and \eqref{def:UFOIC} for $h$ and $A_0$).
\end{theorem}

\begin{figure}[h!]
    \centering
    \subfigure[]{\includegraphics[width=0.4\textwidth]{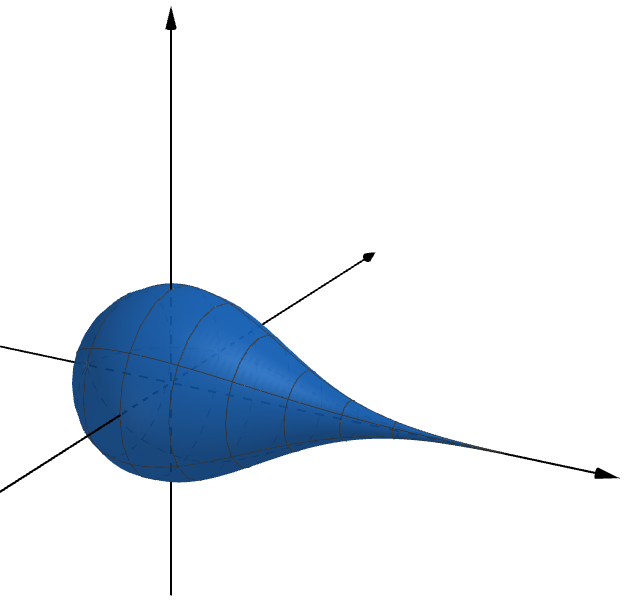}} \hspace{5em}
    \subfigure[]{\includegraphics[width=0.33\textwidth]{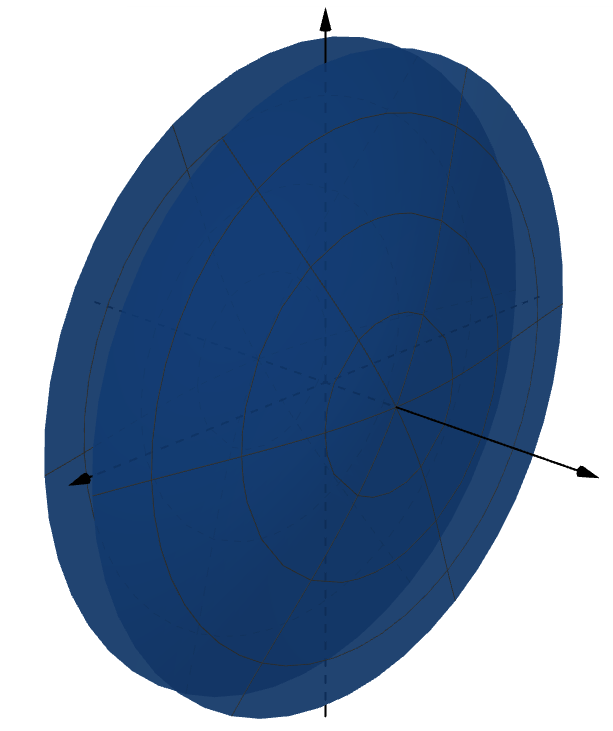}} 
    \caption{(a) Raindrop in $\R^3$ $\qquad\qquad$ (b) UFO in $\R^3$ }
    \label{fig:foobar}
\end{figure}

It will be apparent from the proof of Theorem \ref{thm:main1} that the upper bound on the time of existence weakens as $n$ increases.

\begin{theorem}\label{thm:main2}
    Let $\mathcal{M}_{0}$ be a \emph{stretched torus} hypersurface in $\R^{n+1}$. 
    Then there exists a smooth mean curvature flow
    \[
        F:S^1\times S^{n-1} \times (0,T) \to \R^{n+1},
    \]
    such that $\mathcal{M}_t$ converges to $\mathcal{M}_0$ as $t\to0$ in the sense of \eqref{C}. Additionally, $T\geq\min\{\tfrac{h^2}{2n},\tfrac{\ell^2}{2(n-1)}\}$, where $h$ is the radius of a sphere enclosed by the initial condition and $\ell$ is the distance between $\mathcal{M}_0$ and the axis of rotation (see also Definition \ref{def:stretchedtorusIC} for $h$ and $\ell$). 
\end{theorem}

\begin{figure}[h!]
    \centering
    \includegraphics[width=0.4\linewidth]{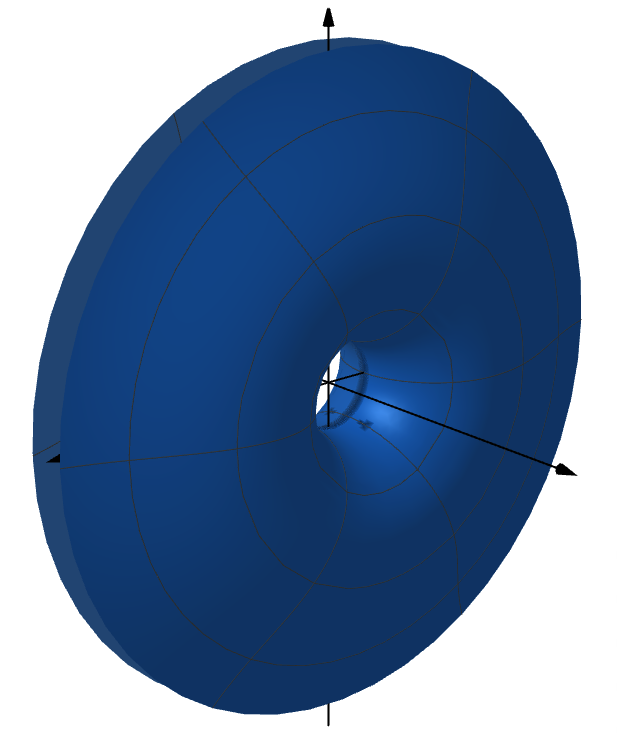}
    \caption{Stretched torus in $\R^3$}
    \label{fig:storusR3}
\end{figure}

\tocless{\section*}{Acknowledgements}

We would like to thank Ben Lambert for useful conversations on the topic. TB and NB were supported by NSF grant DMS-2405007. MF was supported by Discovery Project DP250103808 of the Australian Research Council.

\vspace{\baselineskip}

\section{Preliminaries}

Since we are interested in mean curvature flows of surfaces in $\R^{n+1}$ ($n \geq 2$) whose initial data are surfaces of revolution about the $x$-axis. We start by introducing some terminology and notation concerning surfaces a revolution.

\subsection{Surfaces of Revolution}

Let $\Sigma \in \R^{n+1}$ be a hypersurface which is invariant under rotations about the $x$-axis. If we denote points in $\R^{n+1}$ as $(x,y)$ where $x \in \R$ and $y \in \R^{n}$, then to $\Sigma$ we may associate a curve $\Gamma \subset \R^{2}_{+} \doteqdot \{(x,r) \in \R^{2} : r \geq 0\}$ (and vice-versa) via the following correspondence:
\begin{equation}\label{eq:correspondence}
    (x,r) \in \Gamma \iff (x,y) \in \Sigma \text{ for all } y \in \R^{n} \text{ with } |y| = r. 
\end{equation}
It is clear from this that $\Sigma$ is obtained by rotating the curve $\Gamma$ about the $x$-axis. Therefore, we shall refer to such a $\Sigma$ as a surface of revolution and to the corresponding curve $\Gamma$ as the profile curve.

\begin{defn}
    Let $\Sigma$ be a surface of revolution and let $\Gamma$ be the corresponding profile curve in $\R^{2}_{+}$. 
    \begin{itemize}
        \item[(i)] When an arc, $\Gamma$ can be written as the graph of a function $u:[a,b] \to \R_{+}$ over the $x$-axis, then we refer to the function $u(x)$ as a \emph{horizontal graph function} and its image as a \emph{horizontal graph}.
        \item[(ii)] When an arc, $\Gamma$ can be written as the graph of a function $v:[c,d] \to \R_{+}$ over the $r$-axis, then we refer to the function $v(r)$ as a \emph{vertical graph function} and its image as a \emph{vertical graph}.
    \end{itemize}
\end{defn}

A schematic of horizontal and vertical graphs is visualised in Figure \ref{fig:horizontal_vertical_graphs}. Given a one-parameter family of surfaces of revolution $\{\Sigma_{t}\}_{t \in [0,T)}$, then using the correspondence \eqref{eq:correspondence}, we denote the associated profile curves by $\{\Gamma_{t}\}_{t \in [0,T)}$. Furthermore, we will often switch between discussing the surface $\Sigma_{t}$ and the profile curve $\Gamma_{t}$ under the correspondence \eqref{eq:correspondence}.

\begin{figure}[h]
    \centering
    \includegraphics[width=0.6\linewidth]{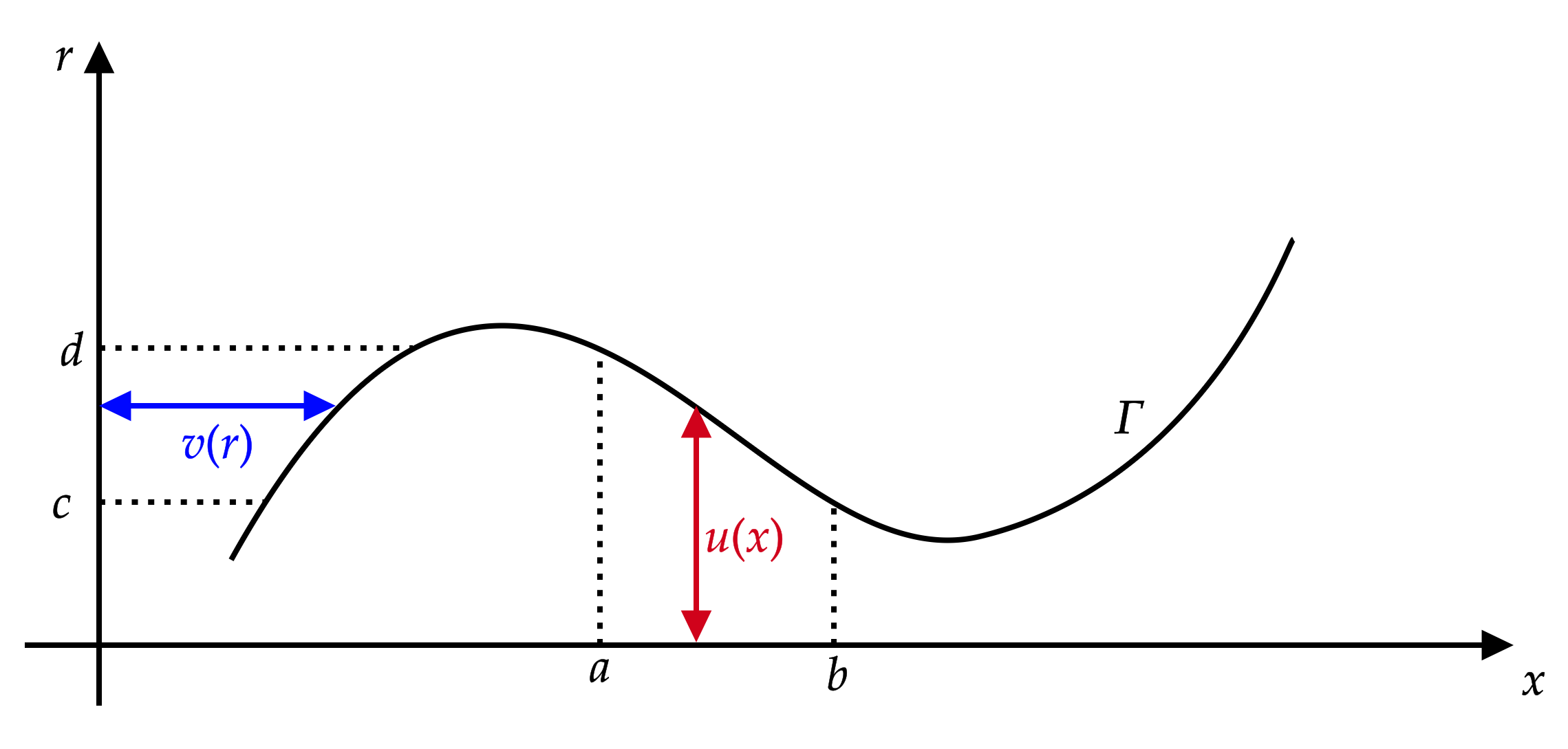}
    \caption{Horizontal and Vertical Graphs}
    \label{fig:horizontal_vertical_graphs}
\end{figure}

If a portion of $\Gamma$ can be represented as a horizontal graph function $u(x)$, then we may write the principal curvatures of $\Sigma$ solely in terms of $u$, $u_{x}$ and $u_{xx}$. To see this parametrise the corresponding portion of $\Sigma$ by $F:[a,b] \times S^{n-1} \to \R^{n+1}$; $(x,\varphi) \mapsto (x,u(x)\varphi)$.  Then, the mean curvature of $\Sigma$, with respect to the unit normal whose restriction on $\Gamma$ coincides with the upward pointing unit normal, is given by 
\begin{equation}\label{eq:mean_curvature_sum1}
    H = \kappa(x) + (n-1)\frac{1}{u(x)\sqrt{1 + u'(x)^{2}}},
\end{equation}
where $\kappa(s)$ is the curvature of the curve $\Gamma$ at the point $(x,u(x))$ with respect to the upward pointing unit normal. 

If instead we can write a portion of $\Gamma$ in terms of a vertical graph function $v(r)$, then we may parametrise the corresponding portion of the surface $\Sigma$ via the function $F:[0,a] \times S^{n-1} \to \R^{n+1}$; $(r,\varphi) \mapsto (r\varphi,v(r))$ and the mean curvature then takes the form
\begin{equation}\label{eq:mean_curvature_sum2}
    H = \kappa(r) - (n - 1)\frac{v'(r)}{r\sqrt{1 + v'(r)^{2}}}
\end{equation}
where again $\kappa(r)$ is the curvature of the curve $\Gamma$ at the point $(r,v(r))$ with respect to the upward pointing unit normal, and also $H$ is regarded with the corresponding normal.

\subsection{Mean Curvature Flow} 

Here we make formal the notation and terminology discussed in the introduction.

\begin{defn}
    Let $F_{0}:M^{n} \to \R^{n+1}$, be a smooth immersion of an $n$-dimensional manifold without boundary $M^{n}$. Then we say that a one-parameter family of smooth immersions $F:M^{n} \times [0,T) \to \R^{n+1}$ is a mean curvature flow with initial data $F_{0}$, if 
    \[
        \begin{cases}
            \displaystyle\frac{\partial F}{\partial t}(x,t) = \vec{H}(x,t) & (x,t) \in M^{n} \times [0,T) \\
            F(x,0) = F_{0}(x) & x \in M^{n}.
        \end{cases}
    \]
    where $\vec{H}(x,t)$ is the mean curvature vector of the immersion $F(x,t)$.
\end{defn}

We will often abuse terminology and refer to the resulting family of immersed hypersurfaces $\{\Sigma_{t}\}_{t \in [0,T)}$, where $\Sigma_{t} \doteqdot F(M^{n},t)$ as a mean curvature flow with initial data $\Sigma_{0}$. In the case that $M^{n}$ is compact, then there exists a unique such flow.

\begin{defn}
    Let $\Omega \subset \R^{n}, I\subset \R$ be  open domains and let $w : \Omega \times I \to \R$ be a one-parameter family of smooth functions, then one may  check that the graphs of $w(\cdot,t)$ evolve by mean curvature flow if and only if 
    \begin{equation}\label{eq:full_graph_equation} 
        \displaystyle\frac{\partial w}{\partial t} = \left(\delta^{ij} - \frac{w_{x_{i}}w_{x_{j}}}{1+ |Dw|^{2}}\right)w_{x_{i}x_{j}}.
    \end{equation}
    Equation \eqref{eq:full_graph_equation} is the standard graphical mean curvature flow equation, but we will refer to it here as the \emph{full graph equation} for reasons that would be clear by the subsequent definitions.
\end{defn}

If $\Sigma_{0}$ is a surface of revolution, then under the mean curvature flow the rotational symmetry will be preserved. Consequently, the corresponding family of hypersurfaces  $\{\Sigma_{t}\}$ can be represented by an evolving family of profile curves $\{\Gamma_{t}\}_{t \in [0,T)}$ in $\R^{2}_{+}$.

\begin{remark}
    Let $\{\Sigma_{t}\}_{t \in [0,T)}$ be a mean curvature flow where each $\Sigma_{t}$ is a surface of revolution.
    \begin{itemize}
        \item[(i)] If the associated profile curves $\Gamma_t$ can be written as a horizontal graph $u(x,t)$ over some fixed interval for a short time,  from \eqref{eq:mean_curvature_sum1} and \eqref{eq:full_graph_equation}, it is easily checked that $u(x,t)$ will satisfy
        \begin{equation}\label{eq:horizontal_graph_equation}
        \displaystyle\frac{\partial u}{\partial t} = \frac{u_{xx}}{1 + u_{x}^{2}} - \frac{n-1}{u}.
        \end{equation}
        We refer to equation \eqref{eq:horizontal_graph_equation} as the \emph{horizontal graph equation}.
        \item[(ii)] If the associated profile curves $\Gamma_{t}$ can be written as a vertical graph $v(r,t)$ over some fixed interval for a short time, from \eqref{eq:mean_curvature_sum2} and \eqref{eq:full_graph_equation}, it is easily checked that $v(r,t)$ will satisfy
        \begin{equation}\label{eq:vertical_graph_equation}
            \displaystyle\frac{\partial v}{\partial t} = \frac{v_{rr}}{1 + v_{r}^{2}} + (n-1)\frac{v_{r}}{r}. 
        \end{equation}
        We refer to equation \eqref{eq:vertical_graph_equation} as the \emph{vertical graph equation}.
    \end{itemize}
\end{remark}
A solution to the vertical graph equation \eqref{eq:vertical_graph_equation} on $[a,b]\times[0,T)$, $a\geq0$, induces a solution to the full graph equation \eqref{eq:full_graph_equation} on $\{y\in\R^{n}:|y|\in[a,b]\}\times[0,T)$ via \[
    w(y,t) = v\left(\sqrt{y_{1}^{2} + \cdots + y_{n}^{2}},t\right).
\]

\subsection{Evolution of Rotationally Symmetric Surfaces} We review some results concerning the evolution of surfaces of revolution. The mean curvature flow of surfaces of revolution 
were extensively studied by Altschuler, Angenent and Giga in \cite{AltschulerAngenentGiga1995}. Their main goal was to analyse the singularities that may develop under the flow and to investigate how the flow can be continued past such singularities. We state here the exact results from their work that we need in this paper. 

In the following, we consider $\Sigma_{0}$ to  be the surface of revolution obtained by rotating a curve $\Gamma_{0}$ about the $x$-axis and denote by $\{\Sigma_{t}\}_{t \in [0,T)}$ the corresponding evolution via mean curvature flow.

\begin{theorem}[\cite{AltschulerAngenentGiga1995} Lemma 5.5]\label{thm:point}
    If $\Gamma_{0}$ can be globally expressed as a horizontal graph $u(x)$ with one global maximum and no other interior critical points, then the family of hypersurfaces $\{\Sigma_{t}\}_{t \in [0,T)}$  converges to a point on the $x$-axis as $t \to T$.
\end{theorem}

This result will play an important role in what follows, as it provides a sufficient condition on the initial surface guaranteeing that the corresponding mean curvature flow develops no singularities before becoming extinct. In particular, this allows us to circumvent certain curvature issues that may arise a priori.

The next result we shall borrow concerns a gradient estimate to the horizontal graph equation \eqref{eq:horizontal_graph_equation}.

\begin{theorem}[\cite{AltschulerAngenentGiga1995} Theorem 4.3]\label{thm:horizontal_gradient}
If $\Gamma_{0}$ can be globally expressed as the graph of a function $r = u(x)$ over the $x$-axis, then so can the profile curves $\Gamma_{t}$ of $\Sigma_{t}$ for each $t \in [0,T)$. Furthermore, if we denote the corresponding solution to the horizontal graph equation \eqref{eq:horizontal_graph_equation} by $r = u(x,t)$ where $u:[\alpha(t),\beta(t)] \times [0,T) \to \R$, then there is a function $\sigma:\R_{+} \times \R_{+} \to \R$ such that 
\[
    |u_{x}(x,t)| \leq \sigma(u(x,t),t),
\]
for all $0 < t < T$, where $\sigma(u,t)$ is non-increasing in both variables and depends only on $\sup_{x \in (\alpha(0),\beta(0))}u(x,0)$.
\end{theorem}

\begin{remark}
    In \cite{AltschulerAngenentGiga1995} it is shown that the estimate in Theorem \ref{thm:horizontal_gradient} holds with $\sigma(u,t) = e^{\rho(u)/t}$ where $\rho(u)$ is a positive non-increasing function on $\R_{+}$ such that $\rho \overset{u \to 0}{\longrightarrow} \infty$. 
\end{remark}

We note that the estimate in Theorem \ref{thm:horizontal_gradient} blows up as $u \to 0$. Consequently, this theorem does not yield a global gradient estimate up to the axis of rotation. As this difficulty will arise in the sequel, we recall the following gradient estimate for the full graph equation due to Evans--Spruck \cite{EvansSpruck1992}, which will play a key role in overcoming it.

\begin{theorem}[\cite{EvansSpruck1992}]\label{thm:evans_spruck}
    If $w \in C^{3}(\overline{\Omega}_{T}) \cap C(\Omega_{T})$ is a solution to the full graph equation \eqref{eq:full_graph_equation} in a domain $\Omega_{T} \doteqdot B_{r}(p) \times (0,T)$, then for all $\epsilon > 0$ there exists a constant $C \doteqdot C(\sup|w|,\epsilon,n,r)$ for which
    \[
        |\nabla w| \leq C
    \]
    in $B_{r-\epsilon}(p) \times (\epsilon^{2},T)$. 
\end{theorem}

 The constant $C$ is increasing in $\sup|w|$ and decreasing in the size of the domain. We remark here that for the purposes of this paper, this bound provides a noticeable advantage to the gradient bound in \cite[Theorem 3.1]{EckerHuisken1991}. The reason is that the latter estimate depends on a gradient bound on the initial surface.

\section{Families of Initial Conditions}

In this section, we define the three types of non-compact initial hypersurfaces that we will consider. We briefly describe their properties and comment on the existence of non-compact mean curvature flow solutions with them as initial conditions. Subsequently, we construct compact approximations of these initial conditions by intersecting them with increasingly larger balls and appropriately capping off the resulting boundaries to obtain compact hypersurfaces. 

\begin{defn}[Raindrop]\label{def:raindropIC}
    Consider a smooth function over the $x$-axis \begin{equation*}
        u:[-a,\infty)\to[0,b],
    \end{equation*} where $2a\geq b>0$, that satisfies the following: 
    \begin{itemize}
        \item[(i)] The graph of $u$ meets the $x$-axis orthogonally at $x = -a$, and the region bounded by the graph of $u$ and the $x$-axis has finite area equal to $\tfrac{A_0}{2}$.
        \item[(ii)] $u$ is monotonically increasing on $(-a,0)$ and monotonically decreasing on $(0,\infty)$.
        \item[(iii)] There exists $h\in (0,a)$, such that the semi-circle $\{(x, \sqrt{h^2-x^2}):x\in[-h,h]\}$ lies strictly under the graph of $u$.
    \end{itemize}
    We will denote the graph of $u$ by 
    \[
        \mathfrak{R} \doteqdot \text{\emph{graph}}(u),
    \]
    and refer to the surface $\mathcal{R}$ obtained by rotating $\mathfrak{R}$ about the $x$-axis as a \emph{raindrop hypersurface} (see Figure \ref{fig:RD_UFO_ICs}).
\end{defn}

Raindrop hypersurfaces $\mathcal{R}$ are connected, non-compact manifolds that are diffeomorphic to $\R^n$, which have unbounded curvature $x \to \infty$. Raindrop hypersurfaces can not arise as the initial condition to a non-compact mean curvature flow, which can be seen by considering Angenent tori which initially surrounds the end of the hypersurface.

\begin{defn}[UFO]\label{def:UFOIC}
    Consider a smooth function over the $r$-axis $$v:[0,\infty)\to[0,d],$$ where $d>0$, that satisfies the following: \begin{itemize}
        \item[(i)] The graph of $v$ meets the $x$-axis orthogonally and the region bounded by the graph of $v$, the $r$-axis and the $x$-axis has finite area equal to $\tfrac{A_0}{4}$.
        \item[(ii)] $v$ is monotonically decreasing on $(0,\infty)$.
        \item[(iii)] There exists $h \in (0,d)$ such that the semi-circle $\{(x, \sqrt{h^2-x^2}):x\in[-h,h]\}$ lies strictly under by the graphs of $v$ and $-v$.
    \end{itemize}
    We will denote the union of the graphs of $v$ and $-v$ by 
    \[
        \mathfrak{U} \doteqdot \text{\emph{graph}}(v) \cup \text{\emph{graph}}(-v)
    \]
    and refer to the surface $\mathcal{U}$ obtained by rotating $\mathfrak{U}$ about the $x$-axis as a \emph{UFO hypersurface} (see again Figure \ref{fig:RD_UFO_ICs}).
\end{defn}

\begin{figure}[h!]
    \centering
    \includegraphics[width=0.9\linewidth]{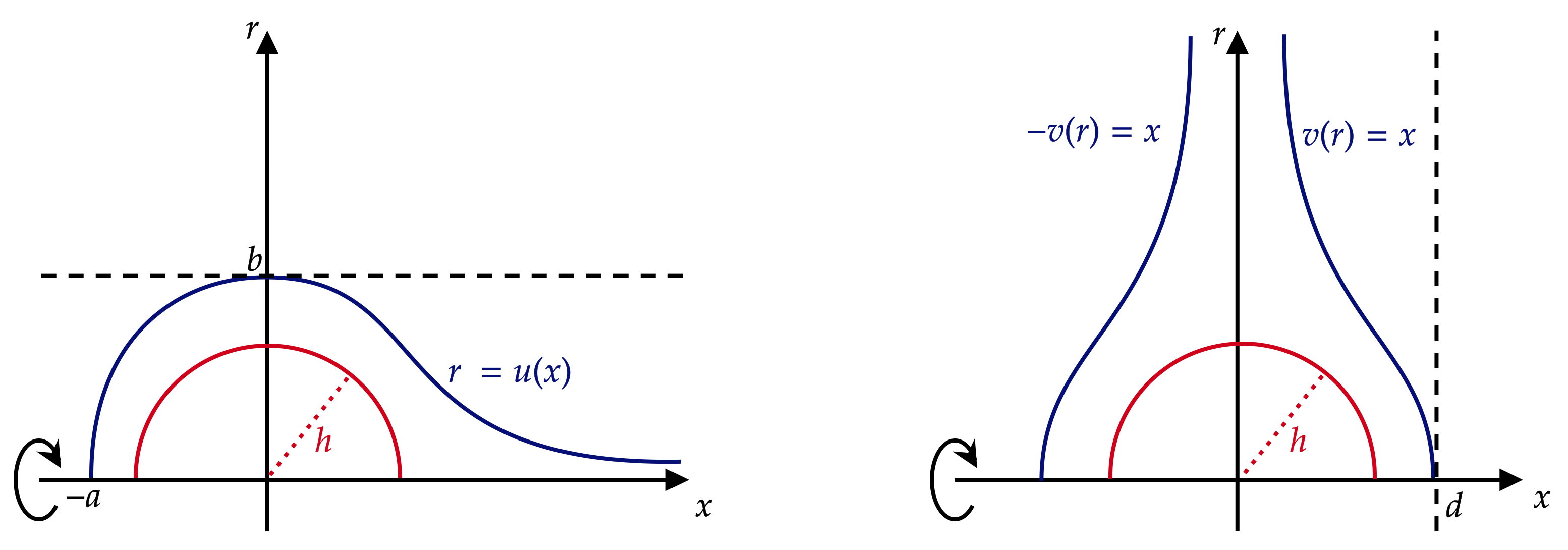}
    \caption{A raindrop (left) and a UFO (right)}
    \label{fig:RD_UFO_ICs}
\end{figure}

UFO hypersurfaces $\mathcal{U}$ are disconnected, non-compact manifolds whose components are diffeomorphic to $\R^n$. They have bounded curvature and are asymptotically flat as $r\to\infty$. Unlike raindrop hypersurfaces, UFO hypersurfaces do arise as initial conditions to a non-compact mean curvature flow. This is because both components of $\mathcal{U}$ may be written as entire Lipschitz graphs, from which the results of \cite{EckerHuisken1991} and \cite{DaskalopoulosSaez2023} imply the existence of a unique non-compact solution for all times.

\begin{defn}[Stretched Torus]\label{def:stretchedtorusIC}
    Consider a smooth function over the $r$-axis $$v:[\ell,\infty)\to[0,L],$$ where $\ell,L>0$, that satisfies the following: \begin{itemize}
        \item[(i)] $v$ meets the $r$-axis orthogonally and the region bounded by the graph of $v$ and the $r$-axis has finite area.
        \item[(ii)] There exists $C>\ell>0$ and $h>\tfrac{L}{2}>0$ such that $x=\sqrt{h^2-(y-C)^2}$ is strictly contained by the graph of $v$ and the $r$-axis.
        \item[(iii)] $v$ is monotonically increasing on $(\ell,C)$ and monotonically decreasing on $(C,\infty)$.
    \end{itemize}
    We will denote the union of the graphs of $v$ and $-v$ by 
    \[
        \mathfrak{T} \doteqdot \text{\emph{graph}}(v) \cup \text{\emph{graph}}(-v)
    \]
    and refer to the surface $\mathcal{T}$ obtained by rotating $\mathfrak{T}$ about the $x$-axis as a \emph{stretched torus hypersurface}.
\end{defn}

\begin{figure}[h!]
    \centering
    \includegraphics[width=0.65\linewidth]{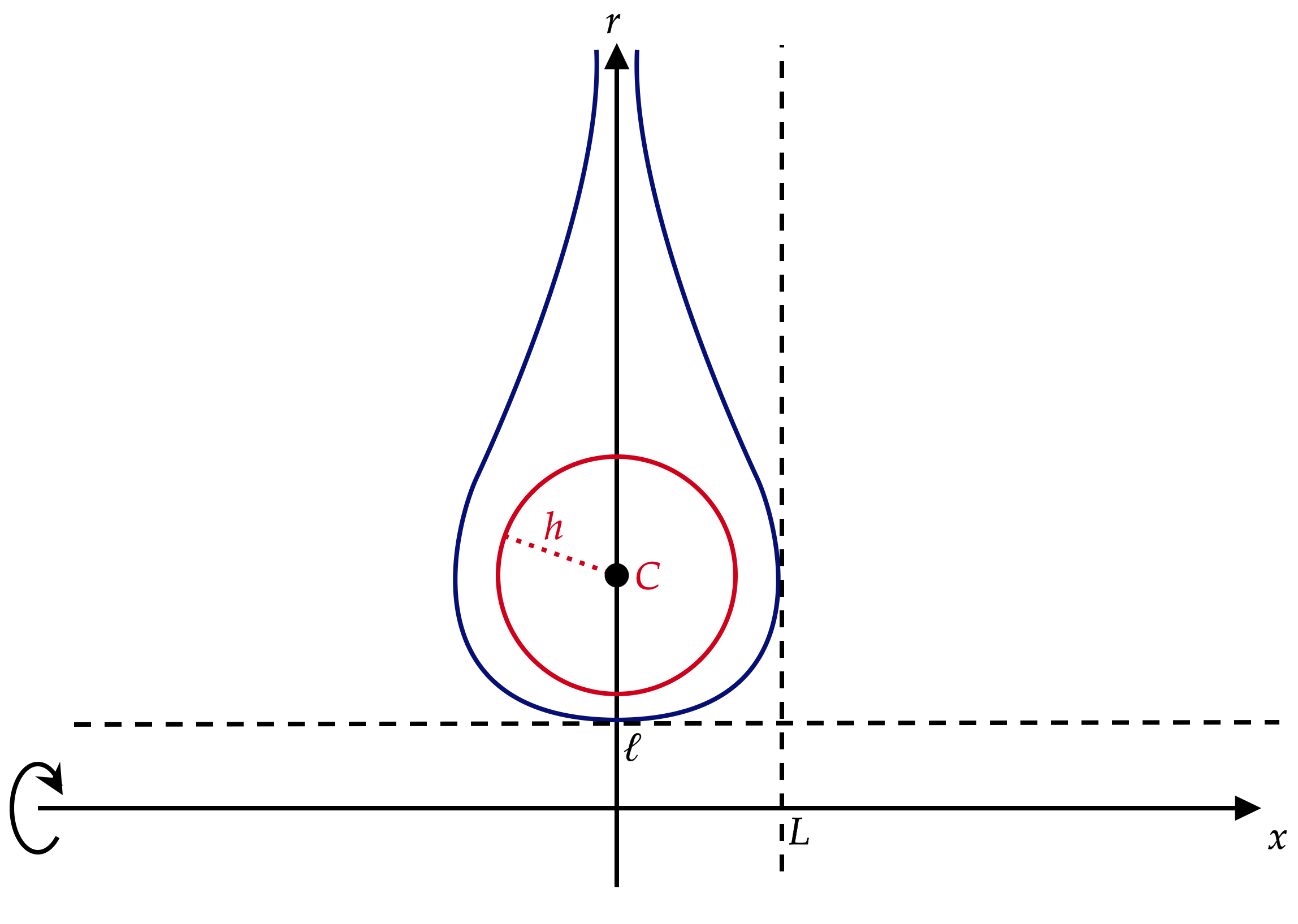}
    \caption{The Stretched Torus}
    \label{fig:StretchedtorusIC}
\end{figure}

The stretched tori are connected, non-compact manifolds with bounded curvature, looking asymptotically flat as $r\to\infty$ and topologically are an $S^{n-1}\times\R$. The stretched tori do arise as the initial condition to a non-compact mean curvature flow since it has bounded second fundamental form \cite[Theorem 4.2]{EckerHuisken1991}.

To establish the existence of compact mean curvature flows arising from these non-compact initial hypersurfaces, we consider suitable compact approximations obtained by truncating the initial hypersurfaces outside balls of increasingly large radius. For each compact approximation, we then consider the corresponding maximal mean curvature flow. The desired mean curvature flow will then be derived by taking a limit as the approximating hypersurfaces converge locally uniformly to the original non-compact hypersurface.

\begin{defn}[Compact Slings]\label{def:slings}
    Suppose $\Gamma$ is either $\mathfrak{R}$, $\mathfrak{U}$ or $\mathfrak{T}$ which contains the sphere $B_{h}(C)$, where $C = 0$ if $\Gamma$ is equal to $\mathfrak{R}$ or $\mathfrak{U}$ and $C > 0$ if $\Gamma = \mathfrak{T}$. For integers $j\geq \max\{2h,C\}$, consider the rectangle $\Box_j\doteqdot[-j,j]\times[0,j]$. Define the compact curve $\Gamma^j$ such that $\Gamma^j\cap\Box_j=\Gamma\cap\Box_j$, $\Gamma^j\cap\Box_{j+1}=\Gamma^j$ and 
    \begin{enumerate}
        \item $\Gamma^j\cap(\Box_{j+1}\setminus\Box_{j})$ can be represented as a monotone graph, over the $x$-axis, in the case that $\Gamma = \mathfrak{R}$,
        \item $\Gamma^{j} \cap (\Box_{j+1} \setminus \Box_{j})$ can be represented as a smooth even graph, over the $x$-axis, with a unique local maximum at $x = 0$ in the case that $\Gamma = \mathfrak{U}$ or $\Gamma = \mathfrak{T}$.
    \end{enumerate}
    We shall refer to the curves $\Gamma^{j}$ as \emph{compact slings}. In keeping with earlier notation, given a compact sling $\mathfrak{R}^{j}$, $\mathfrak{U}^{j}$ or $\mathfrak{T}^{j}$ we will write $\mathcal{R}^{j}$, $\mathcal{U}^{j}$ or $\mathcal{T}^{j}$ for the corresponding (compact) surfaces of revolution. By abuse of notation, we shall also refer to the corresponding compact surfaces of revolution as compact slings.

\end{defn}

The compact slings of the raindrop hypersurfaces \eqref{def:raindropIC} and the UFO hypersurfaces \eqref{def:UFOIC} can be entirely expressed via horizontal graphs with precisely one critical point being a global maximum. As such, by Theorem \ref{thm:point} the mean curvature flow starting from any of these compact slings exists and converge to a point in finite time. Conversely, the compact slings of stretched tori hypersurfaces \eqref{def:stretchedtorusIC} cannot be expressed a horizontal graph and so the long-time existence results of \cite{AltschulerAngenentGiga1995} do not apply, however, the curvature bounds we prove in the next section will prove existence for their compact slings in a uniform time-interval.

\begin{defn}
    We shall refer to the maximal family of hypersurfaces obtained by evolving a compact sling under the mean curvature flow as a \emph{slingshot}. If $\mathcal{R}^{j}$, $\mathcal{U}^{j}$ and $\mathcal{T}^{j}$ denotes a compact sling, then we write the corresponding slingshots as $\{\mathcal{R}^{j}_{t}\}_{t \in [0,T_j)}$, $\{\mathcal{U}^{j}_{t}\}_{t \in [0,T_j)}$ and $\{\mathcal{T}^{j}_{t}\}_{t \in [0,T_j)}$, respectively. Similarly, we shall write $\{\mathfrak{R}_{t}^{j}\}_{t \in [0,T_{j})}$, $\{\mathfrak{U}_{t}^{j}\}_{t \in [0,T_{j})}$ and $\{\mathfrak{T}_{t}^{j}\}_{t \in [0,T_{j})}$ for the corresponding evolving families of profile curves.
\end{defn}

 By Theorem \ref{thm:point} (\cite{AltschulerAngenentGiga1995}, Lemma 5.5), we know that the hypersurfaces corresponding to slingshots $\mathcal{R}_t^j$ and $\mathcal{U}_t^j$ converge to points in finite time. We have no such understanding of the long-time behaviour of the $\mathcal{T}_t^j$ slingshots.

We note that the list of non-compact initial conditions we consider here is by no means exhaustive. Rather, we present some model examples with rotational symmetry.

\section{Curvature Bounds}

In this section, we prove uniform curvature bounds for solutions to \eqref{eq:mcf} with initial data given by a compact sling, as in Definition \ref{def:slings}. For raindrop and UFO hypersurfaces, we utilise two techniques to prove these bounds. We first apply the estimates of Evans-Spruck \cite{EvansSpruck1992} to regions close to the axis of rotation and then the estimates of Altschuler-Angenent-Giga \cite{AltschulerAngenentGiga1995} to regions far away from the axis of rotation. The procedure for the stretched tori hypersurfaces is similar but requires three regions: one region close to the axis of rotation, one region far away from the axis of rotation, and an intermediate region which lies between these two. 

A key step in obtaining uniform curvature bounds over all compact slings arising from any initial condition is ensuring that for all small times $t>0$ they are all uniformly contained in a compact region. This is proven in the curve shortening flow case for certain initial conditions by Bourni and Reiris in  \cite[Lemma 7]{BourniReiris2024}. We recover such result in the case of the compact slings of raindrop and UFO hypersurfaces by observing that their profile curves will move with speed greater than their curvature. Heuristically, we argue that this additional speed will ensure that the compact slings will move into the same compact regions as the ones guaranteed in \cite[Lemma 7]{BourniReiris2024}, if they were to move via curve shortening flow instead.

\begin{lemma}\label{lem:compact1}
    For any decreasing sequence of times $t_j\downarrow0$, there exists a sequence of numbers $x_j$ and $r_j$, such that the following holds. After passing to a subsequence, the  slingshots,  $\mathfrak{R}^j_t$ and $\mathfrak{U}^{j}_t$,  satisfy 
    \begin{equation*}
        \mathfrak{R}^k_t\subset[-a,x_j]\times[0,b],\quad\forall k\geq j,\text{ and }t\geq t_j,
    \end{equation*} and
    \begin{equation*}
        \mathfrak{U}^k_t\subset[-d,d]\times[0,r_j],\quad\forall k\geq j,\text{ and }t\geq t_j,
    \end{equation*}
    respectively.
\end{lemma}

\begin{proof}

    Any solution to the horizontal graph equation in $n$-dimensions is a sub-solution to the graph equation for curves evolving by curve shortening flow; \begin{equation*}
        u_t=\frac{u_{xx}}{1+u_x^2}-\frac{n-1}{u}\leq\frac{u_{xx}}{1+u_x^2}.
    \end{equation*} Therefore, the solution to the horizontal graph equation in $n$-dimensions will be bounded from above by the solution to the curve shortening flow. 
    
    Consider any raindrop or UFO compact sling $\mathfrak{R}^i_0, \mathfrak{U}_{0}^{i}$ at initial time $t=0$. Let $u^{i}(x,t)$ be the corresponding solution to the horizontal graph equation, and let $\tilde{u}(x,t)$ be the corresponding solution to the graphical curve shortening flow with the same initial data. Since $u^{i}(x,0) = 
    \tilde{u}(x,0)$, the maximum principle implies that $u^{i}(x,t) \leq \tilde{u}(x,t)$ for all $t \in [0,T)$. The result now follows from \cite[Lemma 7]{BourniReiris2024}.
    
\end{proof}

\begin{remark}\label{rem:alt_barrier}
    As an alternative barrier for the raindrop slingshots, one could use the complete graphical solutions defined on the punctured disk constructed in \cite{SaezSchnurer2014}, which in the interior of the disk become finite for any positive time (see Figure 3 in \cite{SaezSchnurer2014}).
\end{remark}

Proving an analogous result for the compact slings of the stretched torus initial condition requires more work since in that case the velocity of the profile curves will not move with speed greater than its curvature; a quick way to see this is to notice that in the vertical graph equation \eqref{eq:vertical_graph_equation}, the term $\tfrac{v_r}{r}$ changes sign for the stretched tori compact slings. To prove that these compact slings still enter compact regions in finite time, we utilise UFO compact slings as barriers.

\begin{lemma}\label{lem:compact2}
    For any decreasing sequence of times $t_j\downarrow0$, there exists a sequence of numbers $x_j$, such that, after passing to a subsequence, the slingshots $\mathfrak{T}_t^j$, satisfy \begin{equation*}
        \mathfrak{T}^k_t\subset[-L,L]\times[0,x_j],\quad\forall k\geq j,\text{ and }t\geq t_j.
    \end{equation*}
\end{lemma}

\begin{proof}
    For all stretched tori initial condition $\mathfrak{T}_0$, there exists a UFO initial condition $\mathfrak{U}_0$ which contains it. Therefore, for any $k\in\mathbb{N}$, $\mathfrak{T}_0^k$ is contained inside $\mathfrak{U}_0^{k+1}$. The result then follows by the avoidance principle, together with Lemma \ref{lem:compact1} applied to the flow $\mathfrak{U}^{k+1}_0$.
\end{proof}

In \cite{BourniReiris2024}, uniform curvature bounds are obtained by using a finite covering of ``basic rectangles''; rectangles where straight lines can be used, together with Sturm's theorem, to prove uniform gradient bounds. A similar method applied to the profile curves in higher dimensions would be significantly harder as straight lines are no longer stationary solutions. Instead, we shall make use of the gradient bounds of Evans-Spruck \cite{EvansSpruck1992} and gradient bounds for horizontal graphs, Theorem \ref{thm:horizontal_gradient}.
Let $A$ denote the second fundamental form of an immersed hypersurface.

\begin{proposition}\label{prop:curvrain}
    Suppose $\mathcal{R}_0$ is a raindrop hypersurface satisfying Definition 3.1. Let $0<t_{1} < t_0<\frac{3h^2}{8n}$. Then, for every integer $m\geq0$, there exists a constant $c_m$ depending only on the initial data, $t_{0}$ and $t_{1}$ such that
    \begin{equation}
        \sup_{p\in\mathcal{R}^{k}_t}|\nabla^mA(p)|^2\leq c_m,
    \end{equation} for all $t\in[t_1,t_0]$ and all $k$ sufficiently large. 
\end{proposition}

\begin{proof}
    The choice $0<t_0<\tfrac{3h^2}{8n}$, together with the avoidance principle, ensures that for all $t\in[0,t_0]$ the interior sphere of $\mathcal{R}_{t}^{k}$ (as given in part (iii) of Definition \ref{def:raindropIC}), has radius greater than $\tfrac{h}{2}$.
    Let $t_{1} \in (0,t_{0})$. By construction,   $\mathcal{R}_{t}^{k}\cap\{r=|y| <\tfrac{h}{2}\}$ has two connected components, each of which can be expressed as the  graph of a function of the form
    \begin{equation*}
        w^k(z_1,\dots,z_n,t)\doteqdot v^k(\sqrt{z_1^2+\cdots+z_n^2},t),
    \end{equation*} 
    where $v^k:[0,\tfrac{h}{2}]\times[t_1,t_0]\to\R_+$ is the corresponding solution to the vertical graph equation. The  profile curves $\mathfrak{R}_{t}^{k}$ do remain vertical graphs in the region ${r < h/2}$ for times $t \le t_0$, and therefore the above representation is well defined. 
    
    To see this, notice that by Theorem \ref{thm:horizontal_gradient}, the profile curve can be written globally via a solution to the horizontal graph equation $u^{k}(x,t)$, which at the initial time has only one critical point, namely the maximum. By Sturm's theorem the number of critical points of $u^k(\cdot,t)$ cannot increase in time.
    Consequently, the evolving profile curve continues to have a unique critical point which is necessarily the maximum, and by our choice of $t_0$, the maximum lies outside the region $r<h/2$. Hence the profile curve remains a vertical graph throughout this region. 
    
    The function $w^k$ is a solution to the full graph equation on $\{y\in\R^n:|y|<\tfrac{h}{2}\}\times[t_1,t_0]$ and
    by Lemma \ref{lem:compact1} is uniformly bounded for all sufficiently large $k$. Therefore, the estimates of Evans-Spruck (Theorem \ref{thm:evans_spruck}) provide uniform, in $k$, bounds on all higher derivatives of $w^k$, which depend only on $h$ and $\sup|w^k|$. 
    
    In the region $\{r > \frac{h}{4}\}$, by construction, the horizontal graph functions $u^k(x,t)$ corresponding to the compact slings $\mathfrak{R}^k_{t}$ satisfy $\sup(u^k(x,0),0)\leq b$ for all $k$. Therefore, by Theorem \ref{thm:horizontal_gradient} there exists some $\sigma:\R^+\times\R^+\to\R^+$ depending only on $b$, for which \begin{equation*}
        |u_x^k(x,t)|\leq\sigma(u^k(x,t),t),
    \end{equation*} for all $k$. Recall that $\sigma$ is non-increasing in its arguments, and therefore, for all $k$, $|u^k_x(x,t)|\le \sigma(\tfrac{h}{4},t_1)$, whenever throughout the region $\{r > \frac{h}{4}\}$ and $t\in [t_1,t_0]$. By the standard regularity theory for parabolic PDE's (see \cite[Chapter III, Theorem 12.1]{LadyzhenskayaSolonnikovUraltseva1968}), we obtain uniform, in $k$, bounds on all higher order derivatives of $u^{k}$. Since the second fundamental form of $\mathcal{R}_{t}^{k}$ can be expressed in terms of the functions $w^{k}$ and $u^{k}$, and their derivatives, we obtain the result.
    
\end{proof}

Uniform curvature bounds for the UFO compact slings are obtained in a similar way, with more care required away from the axis of rotation.

\begin{proposition}\label{prop:curvufo}
    Suppose $\mathcal{U}_0$ is a UFO hypersurface satisfying Definition \ref{def:UFOIC}. Let $0<t_{1}<t_0<\frac{3h^2}{8n}$. Then, for every integer $m\geq0$, there exists a constant $c_m$ depending only on the initial data, $t_{0}$ and $t_{1}$ such that\begin{equation}
        \sup_{p\in\mathcal{U}^k_t}|\nabla^mA|^2\leq c_m,
    \end{equation} for all $t\in[t_1,t_0]$ and all $k$ sufficiently large.
\end{proposition}

\begin{proof}

    As in the proof of Proposition \ref{prop:curvrain}, our choice of $t_{0} < \frac{3h^{2}}{8n}$ ensures that for each $\mathcal{U}^{k}_{t}$, the interior sphere has radius greater than $\frac{h}{2}$ for all $t \in [0,t_{0}]$ in accordance with the avoidance principle. Let $t_{1} \in (0,t_{0})$. By construction, $\mathcal{U}_{t}^{k} \cap \{r = |y| < \frac{h}{2}\}$ has two connected components, each of which can be expressed as the graph of a function of the form 
    \[
        w^{k}(z_{1},\cdots,z_{n},t) \doteqdot v^{k}(\sqrt{z^{2}_{1} + \cdots + z^{2}_{n},t}),
    \]
    where $v^{k}:[0,\frac{h}{2}] \times [0,t_{0}] \to \R_{+}$ is the corresponding solution to the vertical graph equation. Indeed, the profile curves $\mathfrak{U}_{t}^{k}$ can be globally written via a solution to the horizontal graph equation $u^{k}(x,t)$ by Theorem \ref{thm:horizontal_gradient}, which at the initial time possesses a single critical point, a maximum at $x = 0$. Therefore, arguing as in Proposition \ref{prop:curvrain}, an application of Sturm's theorem implies that the family of profile curves remain a vertical graph throughout the region $\{r < \frac{h}{2}\}$.

    The function $w^{k}$ is a solution to the full graph equation equation on $\{y \in \R^{n}: |y| < \frac{h}{2}\} \times [0,t_{0}]$, and by construction satisfies $w^{k}(x,t) \leq d$. Therefore, the estimates of Evans-Spruck (Theorem \ref{thm:evans_spruck}) provide uniform bounds on the higher derivatives of $w^{k}$, which depend only on $h$ and $d$.

    In the region $\{r > \frac{h}{4}\}$, by Lemma \ref{lem:compact1} there exists a constant $D > 0$ such that the horizontal graph functions $u^{k}(x,t)$ corresponding to the compact slings $\mathfrak{U}_{t}^{k}$ satisfy $\sup u(x,t_{1}) \leq D$ for sufficiently large $k$. Hence, by Theorem \ref{thm:horizontal_gradient}, there exists some 
    $\sigma:\R^{+} \times \R^{+} \to \R^{+}$ depending only on $D$ for which 
    \[
        |u_{x}^{k}(x,t)| \leq \sigma(u^{k}(x,t),t),
    \]
    for all sufficiently large $k$ and $t \in [t_{1},t_{0}]$. Since $\sigma$ is non-increasing in its arguments, we have that 
    $|u^{k}_{x}(x,t)| \leq \sigma\left(\frac{h}{4},t_{1}\right)$ throughout the region $\{r > \frac{h}{4}\}$ for all sufficiently large $k$ and $t \in [t_{1},t_{0}]$. 
    Precisely as in Proposition \ref{prop:curvrain}, after an application of the regularity theory for parabolic PDE's, we obtain the result.
       
\end{proof}

Finally, we obtain the uniform curvature estimates for the stretched tori.

\begin{proposition}\label{prop:curvtor}
    Suppose $\mathcal{T}_0$ is a stretched torus hypersurface satisfying Definition \ref{def:stretchedtorusIC}. Let $0<t_{1} < t_0<(1-\alpha^{2})\min\{\frac{h^2}{2n},\frac{\ell^2}{2(n-1)}\}$ for any \(\alpha \in (0,1)\). Then, for every integer $m\geq0$, there exists a constant $c_m$ depending only on the initial data, $t_{0}$ and $t_{1}$ such that
    \begin{equation}
        \sup_{p\in\mathcal{T}^k_t}|\nabla^mA|^2\leq c_m,
    \end{equation} for all $t\in[t_1,t_0]$ and all $k$ sufficiently large.
\end{proposition}

\begin{proof}
    
    The choice $t_{0} < (1-\alpha ^2)\min\{\frac{h^{2}}{2n},\frac{\ell^{2}}{2(n-1)}\}$, together with the avoidance principle, ensures that for all $t \in [0,t_{0}]$, the interior sphere of $\mathcal{T}^{k}_{t}$ has radius greater than $\alpha{h}$ and the distance between $\mathcal{T}^{k}_{t}$ and the origin is greater than $\alpha{\ell}$ for all $k$.
    We partition each slingshot $\mathcal{T}_{t}^{k}$ into three regions: the outer region $\{|y| > C + \frac{\alpha h}{2}\}$, the intermediate region $\{C - \alpha h < |y| < C + \alpha h\}$, and the inner region $\{|y| < C - \frac{\alpha h}{2}\}$.

    Arguing as in Propositions \ref{prop:curvrain} and \ref{prop:curvufo}, an application of Sturm's theorem to the corresponding solutions to the horizontal graph equations, $u^{k}_{1}(x,t)$ and $u^{k}_{2}(x,t)$, which lie above and below the regions $\{|y| > C + \alpha h\}$ and $\{|y| < C - \alpha h\}$ respectively, implies that the two components of the profile curves $\mathfrak{T}_{t}^{k}$ remain vertical graphs within the region $\{C-\alpha h < r < C + \alpha h\}$. Therefore, by construction, $\mathcal{T}_{t}^{k} \cap \{C-\alpha h < r < C + \alpha h\}$ has two connected components, each of which can be expressed as the graph of a function of the form 
    \[
        w^{k}(z_{1},\cdots,z_{n},t) \doteqdot v^{k}(\sqrt{z_{1}^{2} + \cdots + z_{n}^{2}},t),
    \]
    where $v^{k}:[C-\alpha h,C+\alpha h] \times [0,t_{0}] \to \R_{+}$ is the corresponding solution to the vertical graph equation. The function $w^{k}$ is a solution to the full graph equation on $\{y \in \R^{n}: C-\alpha h<|y|<C+\alpha h\}$, and by construction satisfies $w^{k}(x,t) \leq L$. Therefore, the estimates of Evans-Spruck (Theorem \ref{thm:evans_spruck}) provide uniform bounds on all the higher order derivatives of $w^{k}$ which depend only on $h, l$ and $L$.

    In the region $\{|y| > C + \frac{\alpha h}{2}\}$, by Theorem \ref{lem:compact2}, there exists a constant $D > 0$ such that the horizontal graph functions $u^{k}_{1}(x,t)$ satisfy $\sup{u(x,t_{1})} \leq D$ for sufficiently large $k$. Hence, by Theorem \ref{thm:horizontal_gradient}, there exists a function $\sigma_{1}:\R_{+} \times \R_{+} \to \R_{+}$ depending only on $D$ for which 
    \[
        |(u_{1}^{k})_{x}(x,t)| \leq \sigma_{1}(u^{k}_{1}(x,t),t),
    \]
    for all sufficiently large $k$ and $t \in [t_{1},t_{0}]$. Therefore, we conclude that $|(u^{k}_{1})_{x}(x,t)| \leq \sigma\left(C+\frac{\alpha h}{2},t_{1}\right)$, throughout the region $\{r > C + \frac{\alpha h}{2}\}$ for all sufficiently large $k$ and $t \in [t_{1},t_{0}]$.

    In the region $\{|y|< C - \frac{\alpha h}{2}\}$, notice that our choice of $t_{0}$ implies that the horizontal graph functions $u^{k}_{2}(x,t)$ satisfy $\alpha\ell \leq u^{k}_{2}(x,t) \leq C - \frac{\alpha h}{2}$ for all $k$ and all $t \in [t_{1},t_{0}]$. Therefore, another application of Theorem \ref{thm:horizontal_gradient} implies the existence of a function $\sigma_{2}:\R_{+} \times \R_{+} \to \R_{+}$ depending only on $C$ and $h$ for which
    \[
        |(u^{k}_{2})_{x}(x,t)| \leq \sigma_{2}(u_{2}^{k}(x,t),t),
    \]
    for all $k$ and $t \in [t_{1},t_{0}]$. Therefore, we conclude that $|(u_{2}^{k})_{x}(x,t)| \leq \sigma_{2}(\alpha\ell,t_{1})$ throughout the region $\{r < C - \frac{\alpha h}{2}\}$ for all $k$ and $t \in [t_{1},t_{0}]$. 

    Precisely as in Propositions \ref{prop:curvrain} and \ref{prop:curvufo}, after an application of the regularity theory for parabolic PDE's, we obtain the result.
    \end{proof}

We are now ready to prove Theorems \ref{thm:main1} and \ref{thm:main2}.

\begin{proof}[Proofs of Theorems \ref{thm:main1} and \ref{thm:main2}]
    We present the proof for the raindrop. The UFO and the stretched torus are proven analogously. 
    
    Let $F^j:S^n\times[0,t_0]\to\R^{n+1}$ denote the embedding such that $F^j(S^n,t)=\mathcal{R}^j_{t}$ for any $t\in(0,t_0)$. 
    Because of the estimates in Propositions \ref{prop:curvrain} and \ref{prop:curvufo}, standard compactness arguments (see, e.g., \cite[§11.1.2]{AndrewsChowGuentherLangford2020}) can be applied to extract a convergent subsequence, yielding a smooth mean curvature flow
    \[
        F^{\infty}: S^{n} \times (0,t_{0}] \to \mathbb{R}^{n+1}\,,
    \]
    such that $F^{\infty}(\cdot,t):S^{n} \to \R^{n+1}$ is a hypersurface of revolution for each $t \in (0,t_{0}]$.
    See the proof of Theorem 1 in \cite{BourniReiris2024} for details on the application of the compactness argument in this setting.
    
    The lower bound on this maximal time $T$ is evident from comparison to the sphere $B_h(0)$ contained in $\mathcal{R}_0$. The upper bound on the time of existence is evident from comparing the evolution of the profile curve under the mean curvature flow to its evolution under the curve shortening flow, as in Lemma \ref{lem:compact1}.

    To prove the closeness statement to the initial condition, let $\varepsilon > 0$.  Let $j > 0$ be the smallest number for which $\mathfrak{R}^{k}_0\cap B_{j}(0)^{c}$ is contained in the strip $\{0 \leq r \leq \varepsilon\}$ for all $k > j$. By using balls of radius $\varepsilon > 0$ as barriers, we can ensure that for all $k > j$, $\emptyset \neq \mathfrak{R}^{k}_{t} \cap B_{j}(0)^{c} \subseteq \{0 \leq r \leq \epsilon\}$ for all $t \in [0,\frac{\varepsilon^{2}}{2n}]$. 
    A similar argument using balls of the same radius $\epsilon$ implies that for all $t \in [0,\tfrac{\epsilon^{2}}{2n}]$,
    \[
        \mathcal{R}^{k}_{t} \cap B_{j}(0) \subset B(\mathcal{R}, \varepsilon).
    \]
    Therefore, for all $k > j$ we have 
    \[
        \mathcal{R}^{k}_{t} \subset B(\mathcal{R},\varepsilon)
    \]
    for all $t \in [0,\frac{\epsilon^{2}}{2n}]$, as required.

\end{proof}

\subsection*{Final Remarks}
We conclude with a few remarks on the usefulness of the examples constructed in this paper and the methods used to construct them.
The solutions constructed here are model rotationally symmetric solutions. In the case of the raindrop and UFO solutions, their evolution is completely understood up to the extinction time. As a result, they have the potential to serve as barriers in a variety of geometric constructions. In fact, we have already demonstrated this idea within the present paper by using the one-dimensional model as a barrier in our constructions and the UFO solution in the construction of the torus.

The general techniques developed in Section 4 to obtain the uniform displacement and curvature estimates are, however, applicable in considerably greater generality. For example, it is conceivable to construct analogous compact solutions emerging from non-compact initial data by attaching raindrop-like fingers to a compact surface. One could then consider a sequence of mean curvature flows with multiple slings (as defined in Definition \ref{def:slings}). The methods developed here can be used to establish short-time existence and to show that, after a uniform time interval, the evolving surfaces are contained in a fixed compact region. To prove that the flow enters a compact region, the one-dimensional model is no longer sufficient. Instead, one can use the barrier constructed in \cite{SaezSchnurer2014}, as discussed in Remark \ref{rem:alt_barrier}. Once the spikes are shown to remain graphical, curvature estimates could be derived using the graphical estimates of Ecker-Huisken \cite{EckerHuisken1989} and Evans-Spruck \cite{EvansSpruck1992}.

\printbibliography

\end{document}